\documentclass[12pt, reqno]{amsart}
\usepackage{amsmath, amsthm, amscd, amsfonts, amssymb, graphicx, xcolor}
\usepackage[bookmarksnumbered, colorlinks, plainpages]{hyperref}
\usepackage{amsmath}
\usepackage{cases}
\usepackage{amssymb}
\usepackage{mathrsfs}
\usepackage{euscript}
\usepackage{dsfont}
\usepackage{graphicx}
\usepackage{float}
\usepackage{graphicx}
\usepackage{amssymb}

\allowdisplaybreaks[4]
\newtheorem{theorem}{Theorem}[section]

\newtheorem{lemma}{Lemma}[section]
\newtheorem{assumption}[lemma]{Assumption}

\newtheorem{corollary}[theorem]{Corollary}
\theoremstyle{definition}

\theoremstyle{remark}

\numberwithin{equation}{section}

\begin{document}
	\setcounter{page}{1}
	
	%\color{darkgray}{
		%\noindent 
		%{\small Annals of Mathematics and Computer Science}\hfill     {\small ISSN: 2789-7206}\\
		%{\small Vol 20 (2024) 1-3}\hfill  {\small https://doi.org/10.56947/amcs.v20.223}}

	%------------------------------------------------------------------------------
	
	%Title of the paper
	\title[]{Convolution-type Bombieri-Vinogradov theorem with well-factorable Weights, and its applications}
	
	\author[]{Zhiyuan Yang}
	
	\address{School of Mathematics,  Shandong University, Jinan 250100, Shandong, China}
	\email{\tt zhiyuan.yang@mail.sdu.edu.cn}

	%Author names and affiliations
	%\author[F. Author, S. Author]{First Author$^1$ and Second Author$^2$$^{*}$}
	
	%\address{$^{1}$ Department of Mathematics, University of California, San Diego, USA.}
	%\email{\textcolor[rgb]{0.00,0.00,0.84}{first@amcs.org}}
	
	%\address{$^{2}$ Department of Computer Science, Moscow State University, Moscow, Russia}
	%\email{\textcolor[rgb]{0.00,0.00,0.84}{second@ieee.org}}

	%\dedicatory{This paper is dedicated to Professor ABCD}
	
	%\date{Received: xxxxxx; Revised: yyyyyy; Accepted: zzzzzz.
		%	\newline \indent $^{*}$ Corresponding author
		%	\newline \indent © The Author(s) 2025. This article is licensed under a Creative Commons Attribution-
		%	\newline \indent NonCommercial-NoDerivatives 4.0
		%	International License. To view a copy of the licence, v%isit 
		%	\newline \indent \url{https://creativecommons.org/licenses/by-nc-nd/4.0/}}
	
	\begin{abstract}
		In this paper, we consider the asymptotic density of $\#\{p\leq x:P^+(p-1)\geq p^c\}$ and $\#\{n\leq x:P^+(n)<P^+(n+1)\}$, where $P^+(n)$ denote the largest prime factor of $n$. We show that for $x\rightarrow\infty$, one has
		\begin{align*}
			\mathop{\lim\sup}_{x\rightarrow\infty} \frac{1}{\pi(x)}\#\{p\leq x:P^+(p-1)\geq p^c\}\leq\frac{16}{5}\log\frac{1}{c}
		\end{align*}
		where $e^{-5/16}<c<1$, and
		\begin{align*}
			\#\{n\leq x:P^+(n)<P^+(n+1)\}>0.299x.
		\end{align*} The first result constitutes an improvement upon that of Ding and Wang (2025), who obatined $\mathop{\lim\sup}_{x\rightarrow\infty} \frac{1}{\pi(x)}\#\{p\leq x:P^+(p-1)\geq p^c\}\leq \frac{7}{2}\log\frac{1}{c}$.
		The second result improves a previous result $0.280$ by the author (2026).
		
		The key to the proof is that for a special class of convolution forms equipped with well‑factorable weights, we may use the level \(x^{5/8-o(1)}\) for Pascadi’s prime‑distribution result with triple-well-factorable weights. We also use Pascadi's estimation of incomplete Kloosterman sums.
		\newline
		\noindent %\textit{2020 Mathematics Subject Classification.} Primary 46L55; Secondary 44B20
	\end{abstract} \maketitle
	
	%Abstract, keywords, math subject classification
	
	\section{Introduction}
	Let $P^+(n)$ and $P^-(n)$ denote the largest and smallest prime factors of $n$ respectively  with the convention $P^+(1)=1$ and $P^-(1)=\infty$.
	\subsection{The upper bound of shifted primes with large prime factors} Define
	\begin{align*}
		T_c(x):=\#\{p\leq x:P^+(p-1)\geq p^c\},
	\end{align*}
	where $1/2<c<1$.
	We are interested in the upper bound of $T_c(x)$.
	We show the previous results as following.
	\begin{align*}
		\text{Erd\H{o}s (1935) [7]}& &&\mathop{\lim\sup}_{x\rightarrow \infty}\frac{T_c(x)}{\pi(x)}\rightarrow0, \quad \text{as}\ c\rightarrow 1\\\text{Ding (2023) [4]}& &&\mathop{\lim\sup}_{x\rightarrow \infty}\frac{T_c(x)}{\pi(x)}\leq\frac{1-2\delta}{2c}, \quad \text{for some}\ \delta>0,\ \text{for}\ 3/4<c<1\\\text{Ding (2025) [5]}& &&\mathop{\lim\sup}_{x\rightarrow \infty}\frac{T_c(x)}{\pi(x)}\leq8(c^{-1}-1),\quad \text{for}\ 8/9<c<1\\\text{Ding and Wang (2025) [6]}& &&\mathop{\lim\sup}_{x\rightarrow \infty}\frac{T_c(x)}{\pi(x)}\leq\frac{7}{2}\log\frac{1}{c},\quad \text{for}\ e^{-2/7}<c<1\\\text{Yang (author, 2026) [32]}& &&\mathop{\lim\sup}_{x\rightarrow \infty}\frac{T_c(x)}{\pi(x)}\leq\frac{1-\delta}{2c}, \quad \text{for some}\ \delta>0,\ \text{for}\ 1/2<c<1
	\end{align*}
	The two \(\delta\) above are derived by different methods and differ from each other. In this paper, we improve Ding and Wang's result.
	\begin{theorem}
		For any $e^{-5/16}<c<1$, we have
		\begin{align*}
			\mathop{\lim\sup}_{x\rightarrow\infty}T_c(x)/\pi(x)&\leq\frac{16}{5}\log\frac{1}{c}.
		\end{align*}
	\end{theorem}
	Note that $16/5<7/2$. The numerical value $e^{-5/16}= 0.73161\cdots$ should be compared with Ding and Wang's result $e^{-2/7}=0.75147\cdots$.

	\begin{corollary}
		For any $ e^{-5/32}<c<1$, we have $\mathop{\lim\sup}_{x\rightarrow\infty}T_c(x)/\pi(x)<1/2$.
	\end{corollary}
	The numerical value of $e^{-5/32}= 0.85534\cdots$ should be compared with Ding and Wang's result $e^{-1/7}=0.86687\cdots$. (see [6, Corollary 1.1])
	
	Under the assumption of Elliott-Halberstam conjecture, we have
	\begin{align*}
		\lim_{x\rightarrow \infty}T_c(x)/\pi(x)=\log\frac{1}{c},
	\end{align*}
	where $1/2<c<1$. 
	Our unconditionally bound in Theorem 1.1 is $3.2$ times that of the expected asymptotic
	density.
	
	\subsection{Erd\H{o}s-Turán conjecture} One of Erd\H{o}s and Turán's
	conjectures (see [22]) asserts that the asymptotic density of integers $n$ satisfying $P^+(n)<P^+(n+1)$ is 1/2. 
	In 1978, Erd\H{o}s and Pomerance [9] first proved that there exists a positive asymptotic density of integers $n$ such that $P^+(n)<P^+(n+1)$. In fact, they showed that\begin{align*}
		\#\{n\leq x:P^+(n)<P^+(n+1)\}>0.0099x.
	\end{align*}
	In 2005, the asymptotic density was improved to 0.05544 by La Bretèche, Pomerance and Tenenbaum [3], and to 0.05866 by Fouvry's arguments in ``Further remarks" of the same paper. Later, the constant 0.05866 was improved successively to 0.1063 and 0.1356 by Wang in [28, 29]. In 2025, Lü and Wang [16] improved this constant to 0.2017. Very recently, the author [32, Theorem 1.4] showed that the asymptotic density is larger than 0.280. In this paper, we improve this constant to 0.299. 
	\begin{theorem}
		For $x\rightarrow\infty$, we have
		\begin{align*}
			\#\{n<x:P^+(n)<P^+(n+1)\}>0.299x.
		\end{align*}
		The lower bound is also true for the pattern $P^+(n)>P^+(n+1)$.
	\end{theorem}
	On the other hand, this problem has generated a long sequence of unconditional density advances, together with logarithmic-density, averaged-shift, and conditional forms [13, 21, 23, 26, 30].
	
	\subsection{The tools and ideas}
	Define $\pi(x) := \#\{\text{prime}\ p \leq x\}$ and $\pi(x, q; a) := \#\{\text{prime}\ p\leq x: p\equiv a\mkern-7mu \pmod{q}\}$. The
	celebrated Bombieri-Vinogradov theorem [1, 27] claim that for $B = B(A)$ large enough in terms of $A$, one has
	\begin{align*}
		\sum_{\substack{q\leq x^{1/2}(\log x)^{-B}\\(q,a)=1}}\Bigg|\pi(x,q,a)-\frac{\pi(x)}{\varphi(q)}\Bigg|\ll_A\frac{x}{(\log x)^A}.
	\end{align*}
	A complex sequence $(\lambda_q)_{q\leq Q}$ is said to be
	well-factorable of level $Q$ iff for any $Q_1$, $Q_2\geq 1$ with $Q_1Q_2=Q$, there exist 1-bounded
	complex sequences $(\alpha_{q_1})$, $(\beta_{q_2})$ supported on $q_i \leq Q_i$, such that for all $q$,
	\begin{align*}
		\lambda_q=\sum_{q_1q_2=q}\alpha_{q_1}\beta_{q_2}.
	\end{align*}
	This weight arises in the Rosser-Iwaniec sieve [12] and has found numerous applications. In 1986, Bombieri, Friedlander and Iwaniec [2] famously obtained 
	\begin{align*}
		\sum_{\substack{q\leq Q\\(a,q)=1}}\lambda_q\Bigg(\sum_{\substack{p<x\\p\equiv a\mkern-15mu\pmod{q}}}1-\frac{1}{\varphi(q)}\sum_{\substack{p<x\\(p,q)=1}}1\Bigg)\ll\frac{x}{(\log x)^A},
	\end{align*}
	where $\lambda_q$ is a well factorable function of level $Q=x^{4/7-\varepsilon}$. 
	
	In our works, we encountered the following type of estimation (This is the key!)
	\begin{align*}
		\sum_{d}\sum_{q}\gamma_d\lambda_q\Bigg(\sum_{\substack{l\sim L}}\sum_{\substack{p<x/l\\lp\equiv a\mkern-15mu\pmod{dq}}}1-\frac{1}{\varphi(dq)}\sum_{\substack{l\sim L}}\sum_{\substack{p<x/l\\(lp,dq)=1}}1\Bigg)\ll_{\varepsilon,A,a}\frac{x}{(\log x)^A},\tag{1.1}
	\end{align*}
	where $L=x^{\nu}$, $(\gamma_d)$ and $(\lambda_q)$ be divisor-bounded sequences supported on positive integers with
	\begin{align*}
		d\sim D=x^{\theta},\quad q\leq Q= x^{\mathcal{L}(\theta,\nu)-\varepsilon},\quad (a,dq)=1,
	\end{align*}
	$\lambda_q$ is a well-factorable function of level $x^{\mathcal{L}(\theta,\nu)-\varepsilon}$ (from the Rosser-Iwaniec sieve).
	In particaular, for $\theta=0$, the above becomes
	\begin{align*}
		\sum_{q}\lambda_q\Bigg(\sum_{\substack{l\sim L}}\sum_{\substack{p<x/l\\lp\equiv a\mkern-15mu\pmod{q}}}1-\frac{1}{\varphi(dq)}\sum_{\substack{l\sim L}}\sum_{\substack{p<x/l\\(lp,q)=1}}1\Bigg)\ll_{\varepsilon,A,a}\frac{x}{(\log x)^A}.\tag{1.2}
	\end{align*}
	In this paper, we mainly focus on the estimation of these two types. The proofs in previous literature mainly employ the following results.
	
	In 1975, Pan, Ding and Wang [18, Theorem 2] showed that for any given positive constant $A>0$, there exists a constant $B=B(A)>0$ such that the estimate\begin{align*}
		\sum_{q\leq x^{1/2}/(\log x)^B}\max_{y\leq x}\max_{(a,q)=1}\Big|\sum_{\substack{L_1<l\leq L_2\\(l,q)=1}}f(l)\Big(\sum_{\substack{lp\leq y\\lp\equiv a\mkern-15mu\pmod{q}}}1-\frac{\mathrm{li}(y/l)}{\varphi(q)}\Big)\Big|\ll_{\varepsilon,A}\frac{x}{(\log x)^A}
	\end{align*}
	holds for $(\log y)^{2B}<L_1\leq L_2<x^{1-\varepsilon}$, where $|f(l)|\leq 1$.
	Their results were later generalized to more general forms. For example, we can see [10, Theorem 9.16].
	Define
	\begin{align*}
		\pi(x;l,a,q):=\sum_{\substack{lp\leq x\\lp\equiv a\mkern-15mu\pmod{q}}}1.
	\end{align*}
	In 2018, Wang [29, Proposition 3.2] generalized Bombieri, Friedlander and Iwaniec's result to show that for any well factorable function $\lambda(q)$ of level $Q$, the following estimate
	\begin{align*}
		\sum_{\substack{q\leq Q\\(a,q)=1}}\lambda(q)\sum_{\substack{L_1\leq l\leq L_2\\(l,q)=1}}\left(\pi(x;l,a,q)-\frac{\pi(x/l)}{\varphi(q)}\right)\ll_{a,A,\varepsilon}\frac{x}{(\log x)^A}
	\end{align*}
	holds for $Q\leq x^{4/7-\varepsilon}$ and $1\leq L_1\leq L_2\leq x^{1-\varepsilon}$.
	Furthermore, there exist many results of a similar convolution type, though they are not written out explicitly in a simple form.
	
	Now let $L=x^{\nu}$ with $0<\nu<3/8$, and consider the expression in (1.2). In our proofs, we observe that almost all \(l \sim L\) have a prime factor $p\in(y^\delta,y)$, where $y\approx x^\varepsilon$, $\delta\rightarrow 0^+$. This is easily verified using the prime number theorem by iteration. (see (3.2) below) Moreover, the indicator function of primes satisfies the so-called Siegel-Walfisz condition. Then Pascadi's estimate for the triply-well-factorable convolution type can be employed. (see Lemma 3.2 below) The triple-well-factorable weights of level $x^{5/8-\varepsilon}$ can be reduced to a well-factorable form, which yields the required result.
	
	Furthermore, when \(\nu>3/8\), we get a better level compared with $x^{5/8-\varepsilon}$. Following Bombieri-Friedlander-Iwaniec's works [2], the key to the proof is the use of Pascadi's new estimate of incomplete Kloosterman sums and a refined version of the Poisson summation formula [see Lemmas 2.2-2.5 below]. In fact, we obtain
	\begin{theorem}
		Let $L=x^{\nu}$. Then for any well-factorable function $\lambda_q$ of level $x^{\mathcal{L}(\nu)-\varepsilon}$, (1.2) holds for
		\begin{align*}
			\frac{3}{8}\leq\nu\leq\frac{1}{2}&:&&\mathcal{L}(\nu)= \frac{1}{4}+\nu,\\\frac{1}{2}\leq\nu\leq1&:&&\mathcal{L}(\nu)= \frac{1}{2}+\frac{\nu}{2}.
		\end{align*}
	\end{theorem}
	As a corollary, we can prove the following result. Let $D=x^{\theta}$ and $L=x^{\nu}$, where $0\leq \theta<\nu$, $3/8\leq \nu\leq 1$. Then for any well-factorable function $\lambda_q$ of level $x^{\mathcal{L}(\nu)-\theta-\varepsilon}$, (1.1) holds. A similar conclusion also holds for $0\leq \theta<\nu<3/8$.

	In fact, the weights in (1.1) are similar to the so-called triply-well-factorable weights, which we recall from [17, Definition 2].  For such weights, Maynard [17, Theorem
	1.1] achieved the exponent of distribution $3/5-\varepsilon$ for primes. Later, Lichtman [14, Corollary 1.5] improved the exponent of distribution for triply-well-factorable
	weights to $66/107-\varepsilon$. The best known result is due to Pascadi [20, Theorem 1.3], who improved 
	the exponents to $5/8-\varepsilon$. We may consider \(\gamma_d\) in (1.1) to be either a smooth function or $1_{(D,2D]}$. Combined with the above results with triply-well-factorable weights, this offers some hope for obtaining improvements for Theorem 1.3.

	\section{Lemmas}

	\begin{lemma}
		We have
		\begin{align*}
			\Phi(x,z):=\sum_{\substack{n\leq x\\ P^-(n)>z}}1=\frac{x\omega(v)-z}{\log z}+O\left(\frac{x}{(\log z)^2}\right),
		\end{align*}
		for $x\geq z\geq 2$, where $\nu=\log x/\log z$, $\omega(v)$ is the Buchstab function.
		For $\varepsilon>0$, we have
		\begin{align*}
			\Psi(x,y):=\sum_{\substack{n\leq x\\P^+(n)\leq y}}1=x\rho(u)\left(1+O_\varepsilon\left(\frac{\log(u+1)}{\log y}\right)\right)
		\end{align*}
		uniformly for
		\begin{align*}
			x\geq x_0(\varepsilon),\qquad \exp\{(\log\log x)^{5/3+\varepsilon}\}\leq y\leq x,
		\end{align*}
		where $u=\log x/\log y$ and $\rho(u)$ is the Dickman function.
	\end{lemma}
	\begin{proof}
		The first result is [25, Chapter III.6., Theorem 6.4].
		The second result is [11, Theorem 1]. In particular, we have $\omega(\nu)\in[1/2,1]$ and $\log \rho(u)=(-1+o(1))u\log u$.
	\end{proof}
	
	\begin{lemma}{\rm (Truncated Poisson with extra steps)}
		Let $x\gg 1$ and $1\ll N,Q\ll x^{O(1)}$, $a\in \mathbb{Z}$, $q\in \mathbb{Z}_+$ with $q\asymp Q$, and $\Phi:(0,\infty)\rightarrow\mathbb{C}$ be a smooth function, $\Phi(t)$ supported in $t\asymp 1$, with $\Phi^{(k)}\ll_k1$ for $k\geq 0$. Then for any $A$, $\delta>0$ and $H:=x^{\delta}N^{-1}Q$, one has\begin{align*}
			\sum_{n\equiv a\mkern-15mu\pmod{q}}\Phi\left(\frac{n}{N}\right)=&\frac{N}{q}\hat{\Phi}(0)+O_{A,\delta}(x^{-A})\\&+\frac{N}{Q}\int\Phi\left(\frac{uq}{Q}\right)\sum_{\substack{H_j=2^j\\1\leq H_j\leq H}}\sum_{h\in\mathbb{Z}}e\left(h\frac{uN}{Q}\right)\Psi_j\left(\frac{|h|}{H_j}\right)e\left(-\frac{ah}{q}\right)\mathrm{d}u,
		\end{align*}
		where $\Psi_j:(1/2,2)\rightarrow\mathbb{C}$ are some compactly supported functions with $\Psi_j^{(k)}\ll_k1$ for $k\geq 0$.
	\end{lemma}
	\begin{proof} 
		This is [20, Lemma 3.1]. Note that the integrand is supported in $u\asymp 1$, and that one can write\begin{align*}
			\frac{N}{Q}\Phi\left(\frac{uq}{Q}\right)\mathrm{d}u=\frac{N}{q}\widetilde{\Phi}\left(\frac{uq}{Q}\right)\frac{\mathrm{d}u}{u},\quad\text{where}\ \widetilde{\Phi}(t):=t\Phi(t).
		\end{align*}
	\end{proof}
	The following assumption is due to Pascadi. For more details, see [20, Assumption 5.4].
	\begin{assumption}{\rm (Exceptional large sieve)}
		We say that a tuple $(q, N, Z, (a_n)_{n \sim N}, A, Y)$, with $q \in \mathbb{Z}_+$, $N \ge 1/2$, $Z \gg 1$, $A \gg \|a_n\|_2$, $Y > 0$, satisfies this assumption iff the following holds. For any $\varepsilon > 0$, $\xi \in \mathbb{R}$, any cusp $\mathfrak{a}$ of $\Gamma_0(q)$ with $\mu(\mathfrak{a}) = q^{-1}$, and any orthonormal basis of exceptional Maass cusp forms $f$ for $\Gamma_0(q)$, with Laplacian eigenvalues $\lambda_f$, $\theta_f := \sqrt{1/4-\lambda_f}$, and Fourier coefficients $\rho_{f\mathfrak{a}}(n)$, one has
		\[
		\sum_{\substack{f \\ \lambda_f < 1/4}}^{\Gamma_0(q)} X^{2\theta_f} \left\vert \sum_{n \sim N} e\left(\frac{n}{N} \xi \right) a_n\, \rho_{f\mathfrak{a}}(n) \right\vert^2 
		\ll_\varepsilon
		(qNZ)^\varepsilon
		\left(1 + \frac{N}{q}\right) 
		A^2,
		\]
		for all 
		\begin{equation} 
			X \ll \max\left(1, \frac{q}{N}\right) \frac{Y}{1+|\xi|^2}.\nonumber
		\end{equation}
	\end{assumption}
	
	Define
	\begin{align*}
		T_H(\alpha):=\min_{t\in\mathbb{Z}_+}(t+H\|t\alpha\|),
	\end{align*}
	where $\|\alpha\|:= \min_{n\in\mathbb{Z}} |\alpha-n|$. We have $T_H(\alpha)\ll 1+H|\alpha|$.
	\begin{lemma}
		Let $N \ge 1/2$, $L, H \gg 1$, $\alpha_1, \alpha_2 \in \mathbb{R}/\mathbb{Z}$, and $q, \ell_1, \ell_2 \in \mathbb{Z}_+$, $a \in \mathbb{Z}$ be such that $q \gg L^2$, $\ell_1, \ell_2 \asymp L$, and $(\ell_1, \ell_2) = 1$. Let $\Phi_i(t) : (0, \infty) \to \mathbb{C}$ be smooth functions supported in $t \ll 1$, with $\Phi_i^{(j)} \ll_j 1$ for all $j \ge 0$, and
		\[
		a_n := \sum_{\substack{h_1, h_2 \in \mathbb{Z} \\ a(h_1 \ell_1 - h_2 \ell_2) = n}} \Phi_1\left(\frac{h_1}{H}\right) \Phi_2\left(\frac{h_2}{H}\right) e(h_1 \alpha_1 + h_2 \alpha_2).
		\]
		Then the tuple $(q, N, H, (a_n)_{n \sim N}, A, Y)$ satisfies Assumption 2.3, where
		\[
		A := \|a_n\|_2 + \sqrt{N \left(\frac{H}{L} + \frac{H^2}{L^2}\right)}, \qquad\qquad Y := \max\left(1, \frac{NH}{|a|(H+L)L\min_i T_H(\alpha_i)}\right).
		\]
	\end{lemma}
	\begin{proof}
		This follows from [20, Theorem 1.7]. (Also see [19, Proposition 3.4])
	\end{proof}
	
	\begin{lemma}
		Let $R$, $S$, $N\geq 1/2$, $C$, $D$, $Z\gg1$, and $Y$, $\varepsilon>0$. Let $\theta_{\max}=7/64$. For all $r\sim R$, $s\sim S$ with $(r,s)=1$, let:
		\begin{itemize}
			\item $w_{r,s} \in \mathbb{C}$;
			\item $\Phi_{r,s} : (0,\infty)^3 \to \mathbb{C}$ be smooth, with $\Phi_{r,s}(x,y,z)$ supported in $x,y,z \asymp 1$, and
			\[
			\partial_x^j \partial_y^k \partial_z^\ell \Phi_{r,s}(x,y,z) \ll_{j,k,\ell,\varepsilon} Z^{j\varepsilon},
			\quad \forall j,k,\ell \geq 0;
			\]
			\item $(rs, N, Z, (a_{n,r,s})_{n \sim N}, A_{r,s}, Y)$ be a tuple satisfying Assumption 2.3.
		\end{itemize}
		Then with a consistent choice of the sign $\pm$, it holds that\begin{align*}
			&\sum_{\substack{r\sim R\\s\sim S\\(r,s)=1}}w_{r,s}\sum_{n\sim N}a_{n,r,s}\sum_{\substack{c,d\\(rd,sc)=1}}\Phi_{r,s}\left(\frac{n}{N},\frac{d}{D},\frac{c}{C}\right)e\left(\pm n\frac{\overline{rd}}{sc}\right)\\&\ll_{\varepsilon}(RSNCDZ)^{O(\varepsilon)}||w_{r,s}A_{r,s}||_2\mathscr{I},
		\end{align*}
		where \begin{align*}
			\mathscr{I}^2:=D^2NR+\left(1+\frac{C^2}{R^2SY}\right)^{2\theta_{\max}}CS(C+DR)(RS+N).
		\end{align*}
	\end{lemma}
	\begin{proof}
		This is [20, Corollary 5.14].
	\end{proof}
	
	\begin{lemma}
		Let $D\geq 2$ and $L>1$. Let $\mathscr{P}$ denote a set of primes. Let $z\geq 2$ and write $P(z):=\prod_{p\leq z,p\in\mathscr{P}}p$. There exist two sequences $\{\lambda_d^\pm\}_{d=1}^{\infty}$ of real numbers, vanishing for $d>D$ or $\mu(d)=0$, satisfying $\lambda_1^\pm=1$, $|\lambda_d^{\pm}|=O(1)$, $\lambda^-*1\leq \mu*1\leq \lambda^+*1$, and such that
		\begin{align*}
			&\sum_{d|P(z)}\lambda_d^+\frac{\omega(d)}{d}\leq \prod_{\substack{p\leq z\\p\in \mathscr{P}}}\left(1-\frac{\omega(p)}{p}\right)\left(F(s)+o(1)\right),\\&\sum_{d|P(z)}\lambda_d^-\frac{\omega(d)}{d}\geq \prod_{\substack{p\leq z\\p\in \mathscr{P}}}\left(1-\frac{\omega(p)}{p}\right)\left(f(s)+o(1)\right)
		\end{align*}
		uniformly for all multiplicative function $\omega$ satisfying
		\begin{align*}
			&(i)\ 0<\omega(p)<p\ (p\in\mathscr{P}),\\&(ii)\ \prod_{\substack{u<p\leq v\\p\in\mathscr{P}}}\left(1-\frac{\omega(p)}{p}\right)^{-1}\leq \frac{\log v}{\log u}\left(1+\frac{L}{\log u}\right)\ (2\leq u\leq v\leq z),
		\end{align*}
		where $s=\log D/\log z=O(1)$ and $f(s)$, $F(s)$ are determined by the differential-difference equations
		\begin{equation}
			\begin{aligned}
				&\left\{ \begin{aligned}
					&sF(s)=2e^{\gamma}, \quad&& 1\leq s\leq 2;\\&sf(s)=0,\quad&& 0\leq s\leq 2;\\&(sF(s))'=f(s-1),(sf(s))'=F(s-1),\quad&&s\geq 2.
				\end{aligned}\right.\nonumber
			\end{aligned}
		\end{equation}
		In particular, the functions $\lambda^{\pm}$ can be rewrited as
		\begin{align*}
			\lambda^{\pm}=\sum_{h<\exp(8\varepsilon^{-3})}\lambda^{\pm}(\cdot,h),
		\end{align*}
		where $\lambda^{\pm}(\cdot,h)$ is a well-factorable function of level $D$.
	\end{lemma}
	\begin{proof}
		This lemma is the Rosser-Iwaniec sieve [12].
	\end{proof}  
	
	\begin{lemma}
		For any given positive constant $A>0$, there exists a constant $B=B(A)>0$ such that the estimate\begin{align*}
			\sum_{q\leq x^{1/2}/(\log x)^B}\max_{y\leq x}\max_{(a,q)=1}\Big|\sum_{\substack{L_1<l\leq L_2\\(l,q)=1}}f(l)\Big(\sum_{\substack{lp\leq y\\lp\equiv a\mkern-15mu\pmod{q}}}1-\frac{\mathrm{li}(y/l)}{\varphi(q)}\Big)\Big|\ll\frac{x}{(\log x)^A}
		\end{align*}
		and
		\begin{align*}
			\sum_{q\leq x^{1/2}/(\log x)^B}\max_{y\leq x}\max_{(a,q)=1}\Big|\sum_{\substack{L_1<l\leq L_2\\(l,q)=1}}f(l)\Big(\sum_{\substack{p\leq y/L_2\\lp\equiv a\mkern-15mu\pmod{q}}}1-\frac{\mathrm{li}(y/L_2)}{\varphi(q)}\Big)\Big|\ll\frac{x}{(\log x)^A}
		\end{align*}
		hold for $(\log y)^{2B}<L_1\leq L_2<x^{1-\varepsilon}$, where $|f(l)|\leq 1$ and the constant implied by the symbol ``$\ll$" depends only on $\varepsilon$ and $A$.
	\end{lemma}
	\begin{proof}
		This is [16, Lemma 2.3] (see [18, Theorem 2]). 
	\end{proof}
	
	\begin{lemma}
		Define
		\begin{align*}
			H(n):=\prod_{p>2,p|n}\left(\frac{p-1}{p-2}\right),
		\end{align*}
		Then for $x\gg a$ and $y>x^{\varepsilon}$, we have
		\begin{align*}
			&\sum_{\substack{n\leq x\\(n,a)=1}}H(n)=\frac{cx}{2^{\varepsilon(a)}H(a)}(1+o(1)),\\&\sum_{\substack{n\leq x,P^+(n)\leq y\\(n,a)=1}}H(n)=\frac{c\Psi(x,y)}{2^{\varepsilon(a)}H(a)}(1+o(1)),
		\end{align*}
		where 
		\begin{align*}
			c=\prod_{p>2}\left(1+\frac{1}{p(p-2)}\right).
		\end{align*}
	\end{lemma}
	\begin{proof}
		We can see it in [3, p.136] and [15, (3.5)].
	\end{proof}
	\subsection*{Convention.} We use $\varepsilon$ to denote a sufficiently small positive number, and the value of $\varepsilon$ may change from statement to statement. We use $\mu(n)$ and $\varphi(n)$ to denote the M$\mathrm{\ddot{o}}$bius function and Euler's function, respectively. We use $\omega(\nu)$ and $\rho(u)$ to denote the Buchstab function and Dickman function, respectively. By $(m_1,m_2)$ we denote the greatest common divisor of $m_1$ and $m_2$. We use the standard asymptotic notation $f\ll g,f\gg g,f=O(g),f=o_{x\rightarrow\infty}(g)$ from analytic number theory, and indicate that the implicit constants depend on some parameter $\varepsilon$ through subscripts. By $m\sim M$ we denote $M< m\leq2M$. Statements like $f(x)\ll x^{o(1)}g(x)$ should be read as $\forall \varepsilon>0$, $f(x)\ll_\varepsilon x^\varepsilon g(x)$. Define
	\begin{align*}
		P(x):=\prod_{p\leq x}p,\quad P(y,z):=\prod_{y<p\leq z}p.
	\end{align*}

	\section{Proof of Theorem 1.1}
	We use $\varepsilon$ to denote a suffciently small positive number, and the value of $\varepsilon$ may change from statement to statement. Fix $e^{-5/16}<c<1$.
	We apply the following lemma to estimate $T_c(x)$.\begin{lemma}
		For $0<c<1$ and sufficiently large $x$, we have
		\begin{align*}
			\sum_{\substack{p\leq x\\P^+(p-1)\geq p^c}}1=\sum_{\substack{p\leq x\\P^+(p-1)\geq x^{c}}}1+O\left(\frac{x\log\log x}{(\log x)^2}\right).
		\end{align*}
	\end{lemma}\begin{proof}
		This is [31, Theorem 2].
	\end{proof}
	We start from the following expression
	\begin{align*}
		T_c'(x):=\sum_{\substack{p'\leq x\\P^+(p'-1)\geq x^c}}1=\sum_{x^c\leq p<x}\sum_{\substack{\ell p\leq x\\\ell p+1\ \text{is prime}\\2|\ell}}1\leq \sum_{\substack{ \ell\leq x^{1-c}\\2|\ell }}\sum_{\substack{\ell p\leq x\\\ell p+1\ \text{is prime}}}1.
	\end{align*}
	Fix $0<\delta<1/100$. Let $y=x^{(1-c)\varepsilon/(3J)}$, where $J=J(\delta)$ is a parameter satisfying
	\begin{align*}
		\frac{(\log(1/\delta))^{J+1}}{(J+1)!}\leq \delta.
	\end{align*}
	Since
	\begin{align*}
		\sum_{\ell}=\sum_{\substack{\ell\\(\ell,P(y^\delta,y))>1}}+\sum_{\substack{\ell\\(\ell,P(y^\delta,y))=1}},
	\end{align*}
	we write
	\begin{align*}
		T_c'(x)\leq \sum_{\substack{ \ell\leq x^{1-c}\\(\ell,P(y^\delta,y))>1\\2|\ell }}\sum_{\substack{\ell p\leq x\\\ell p+1\ \text{is prime}}}1+\sum_{\substack{ \ell\leq x^{1-c}\\(\ell,P(y^\delta,y))=1\\2|\ell }}\sum_{\substack{\ell p\leq x\\\ell p+1\ \text{is prime}}}1=\mathscr{S}_1+\mathscr{S}_2.\tag{3.1}
	\end{align*}
	\subsection{Estimate of $\mathscr{S}_1$}
	
	Note that for any real sequence $(a_\ell)$ with $a_{\ell}\geq 0$, we have
	\begin{align*}
		\sum_{\substack{(\ell,P(y^\delta,y))>1}}a_{\ell}&=\sum_{y^\delta<p_1\leq y}\sum_{\substack{(\ell,P(y^\delta,y))=1}}a_{p_1\ell}+\frac{1}{2!}\sum_{y^\delta<p_1,p_2\leq y}\sum_{\substack{(\ell,P(y^\delta,y))=1}}a_{p_1p_2\ell}\\&\quad+\frac{1}{3!}\sum_{y^\delta<p_1,p_2,p_3\leq y}\sum_{\substack{(\ell,P(y^\delta,y))=1}}a_{p_1p_2p_3\ell}+\cdots+O\Bigg(\sum_{y^{\delta/2}<p\leq y^{1/2}}\sum_{\ell}a_{p^2\ell}\Bigg)\\&\leq \sum_{1\leq j\leq J}\frac{1}{j!}\sum_{y^\delta<p_1,\cdots,p_j\leq y}\sum_{\substack{(\ell,P(y^\delta,y))=1}}a_{p_1\cdots p_j\ell}+\frac{1}{(J+1)!}\sum_{y^\delta<p_1,\cdots,p_{J+1}\leq y}\sum_{\substack{p_1\cdots p_{J+1}|\ell}}a_{\ell}\\&\quad+O\Bigg(\sum_{y^{\delta/2}<p\leq y^{1/2}}\sum_{\ell}a_{p^2\ell}\Bigg).\tag{3.2}
	\end{align*}
	Let $D=z^2=x^{5/8-\varepsilon}$ and $P(z)=\prod_{p\leq z}p$.
	Define
	\begin{align*}
		\mathcal{L}:=\{\ell\leq x^{1-c}:(\ell,P(y^\delta,y))=(\ell,P^2(y^\delta,y))>1,2|\ell\}.
	\end{align*}
	By Lemma 2.6, there exists a sequence
	 $\{\lambda_{d}^+\}_{d=1}^{\infty}$ of real numbers, vanishing for $d>D$ or $\mu(d)=0$, satisfying $\lambda_1^+=1$, $|\lambda_{d}^+|=O(1)$, $0\leq \mu*1\leq \lambda^+*1$. In particular, this sequence $\lambda^+$ can be expressed as a sum of well-factorable sequences.
	 We have
	\begin{align*}
		\mathscr{S}_1&\leq  \sum_{\substack{ \ell\in\mathcal{L}}}\sum_{\substack{\ell p\leq x\\(\ell p+1,P(z))=1}}1+O(z)+O\Bigg(\sum_{y^{\delta/2}<p_1\leq y^{1/2}}\sum_{\substack{p_1\leq x\\p\equiv 1\mkern-15mu\pmod{p_1^2}}}1\Bigg)\\&\leq \sum_{\substack{ \ell\in\mathcal{L} }}\sum_{\substack{\ell p\leq x}}\sum_{d|(\ell p+1,P(z))}\lambda_d^++o(\pi(x))\\&\leq \sum_{1\leq j\leq J}\frac{1}{j!}\sum_{y^\delta<p_1,\cdots,p_j\leq y}\sum_{\substack{ \ell\leq x^{1-c}\\p_1\cdots p_j|\ell\\(\ell(p_1\cdots p_j)^{-1},P(y^\delta,y))=1\\2|\ell }}\sum_{\substack{\ell p\leq x}}\sum_{d|(\ell p+1,P(z))}\lambda_d^+\\&\quad +\frac{1}{(J+1)!}\sum_{y^\delta<p_1,\cdots,p_{J+1}\leq y}\sum_{\substack{ \ell\leq x^{1-c}\\p_1\cdots p_{J+1}|\ell\\2|\ell }}\sum_{\substack{\ell p\leq x}}\sum_{d|(\ell p+1,P(z))}\lambda_d^++o(\pi(x)).\tag{3.3}
	\end{align*}
	For any $1\leq j\leq J$, changing the order of summation, we write
	\begin{align*}
		&\sum_{y^\delta<p_1,\cdots,p_j\leq y}\sum_{\substack{ \ell\leq x^{1-c}\\p_1\cdots p_j|\ell\\(\ell(p_1\cdots p_j)^{-1},P(y^\delta,y))=1\\2|\ell }}\sum_{\substack{\ell p\leq x}}\sum_{d|(\ell p+1,P(z))}\lambda_d^+\\&=\sum_{d|P(z)}\lambda_d^+\sum_{y^\delta<p_1,\cdots,p_j\leq y}\sum_{\substack{ \ell\leq x^{1-c}\\p_1\cdots p_j|\ell\\(\ell(p_1\cdots p_j)^{-1},P(y^\delta,y))=1\\2|\ell }}\sum_{\substack{\ell p\leq x\\(\ell p,d)=1}}\frac{1}{\varphi(d)}\\&\quad+\sum_{d|P(z)}\lambda_d^+\sum_{y^\delta<p_1,\cdots,p_j\leq y}\sum_{\substack{ \ell\leq x^{1-c}\\p_1\cdots p_j|\ell\\(\ell(p_1\cdots p_j)^{-1},P(y^\delta,y))=1\\2|\ell }}\sum_{\substack{\ell p\leq x}}\Bigg(1_{\ell p\equiv -1\mkern-15mu\pmod{d}}-\frac{1_{(\ell p,d)=1}}{\varphi(d)}\Bigg)\\&=M_j+R_j.
	\end{align*}
	Write $\ell=p_1\cdots p_j\ell_1$ with $(\ell_1,P(y^c,y))=1$.
	In order to prove $R_j\ll x(\log x)^A$, putting $p_1,e=p_2\cdots p_j \ell_1$ and $p$ in dyadic range, we only need to prove that for any $A>1$
	\begin{align*}
		\sum_{q}\lambda_q\sum_{\substack{p_1\in\mathscr{P}_1}}\sum_{e\in\mathscr{E}}b_e\sum_{p\in\mathscr{P}}\Bigg(1_{p_1ep\equiv a\mkern-15mu\pmod{q}}-\frac{1_{(p_1ep,q)=1}}{\varphi(q)}\Bigg)\ll\frac{x}{(\log x)^A},\tag{3.4}
	\end{align*}
	where $\lambda_q$ is well factorable function of level $D$, $|b_e|\leq 1$, $\mathscr{P}_1=[(1-\Delta)P_1,P_1]$, $\mathscr{E}=[(1-\Delta)E,E]$, $\mathscr{P}=[(1-\Delta)P,P]$ with $P_1EP\asymp x$, $P_1E\ll x^{1-c}$, $y^\delta<P_1\leq y$, $\Delta=(\log x)^{-A_1}$. ($A_1$ is a suitable constant.)
	
	We need the following lemma due to Pascadi. (see [19, Proposition 4.4])
	\begin{lemma}
		Let $a\in\mathbb{Z}\setminus\{0\}$, $A,\varepsilon>0$, and $M,N,x,Q_1,Q_2,Q_3\gg1$ satisfy $MN\asymp x$, $N>x^\varepsilon$. 
		Let $(\alpha_n)$, $(\beta_m)$ be $1$-bounded complex sequences, such that $(\alpha_n)$ is supported on $P^-(n)\geq z_0:=x^{1/(\log\log x)^3}$ and satisfies the Siegel-Walfisz condition, which means
		\begin{align*}
			\Bigg|\sum_{\substack{n\sim N\\n\equiv b\mkern-15mu\pmod{q}\\(n,d)=1}}\alpha_n-\frac{1}{\varphi(q)}\sum_{\substack{n\sim N\\(n,dq)=1}}\alpha_n\Bigg|\ll_A\frac{N\tau(d)^{O(1)}}{(\log N)^A},
		\end{align*}
		for any $d\geq 1, q\geq 1$ and $(b,q)=1$, $A>1$.
		If 
		\begin{align*}
			Q_1&\leq Nx^{-\varepsilon} ,\\N^2Q_2Q_3^2&\leq x^{1-15\varepsilon},\\N^{2}Q_2^5Q_3^2&\leq x^{2-40\varepsilon},
		\end{align*} then for any $1$-bounded complex sequences $(\gamma_{q_1})$, $(\lambda_{q_2})$, $(\nu_{q_3})$ supported on $(q_i,a)=1$, one has\begin{align*}
			\sum_{q_1\sim Q_1}\gamma_{q_1}\sum_{q_2\sim Q_2}\lambda_{q_2}\sum_{q_3\sim Q_3}\nu_{q_3}\sum_{n\sim N}\alpha_n\sum_{m\sim M}\beta_m\Bigg(1_{mn\equiv a\mkern-15mu\pmod{q}}-\frac{1_{(mn,q)=1}}{\varphi(q)}\Bigg)\ll_{\varepsilon,A,a}\frac{x}{(\log x)^A}.
		\end{align*}
	\end{lemma}
	
	Applying Lemma 3.2 with $Q_1=1$, $\varepsilon\leftarrow(1-c)\delta\varepsilon/(3J)$, $\alpha_n=1_{n\ \text{is prime}}1_{ n\in \mathscr{P}_1}$,
	we obtain that (3.4) holds,
	by writing $\lambda=\lambda_2*\lambda_3$ where $\lambda_2(q_2)$ is supported on $q_2\leq x^{1/4-\varepsilon}$, $\lambda_3(q_3)$ is supported on $q_3\leq x^{3/8}$. We have controlled the remainder term.
	
	Now we begin to estimate the main term. We have
	\begin{align*}
		M_j&=\sum_{d|P(z)}\lambda_d^+\sum_{y^\delta<p_1,\cdots,p_j\leq y}\sum_{\substack{ \ell\leq x^{1-c}(p_1\cdots p_j)^{-1}\\(\ell,P(y^\delta,y))=1\\2|\ell }}\sum_{\substack{p_1\cdots p_j\ell p\leq x}}\frac{1_{(p_1\cdots p_j\ell p,d)=1}}{\varphi(d)}\\&=\sum_{y^\delta<p_1,\cdots,p_j\leq y}\sum_{\substack{ \ell\leq x^{1-c}(p_1\cdots p_j)^{-1}\\(\ell,P(y^\delta,y))=1\\2|\ell }}\pi(x/(p_1\cdots p_j\ell))\sum_{d|P'(z)}\frac{\lambda_d^+}{\varphi(d)},
	\end{align*}
	where $P'(z)=\prod_{p<z,(p,p_1\cdots p_j\ell)=1}p$. By Lemma 2.6, we have
	\begin{align*}
		\sum_{d|P'(z)}\frac{\lambda_d^+}{\varphi(d)}&\leq (F(2)+o(1))\prod_{\substack{2<p\leq z\\(p,p_1\cdots p_j\ell)=1}}\left(1-\frac{1}{p-1}\right)\\&=(e^\gamma+o(1))\prod_{\substack{2<p\leq z}}\left(\frac{p-2}{p-1}\right)\prod_{p|\ell}\left(\frac{p-1}{p-2}\right)
	\end{align*}
	Mertens' theorem implies
	\begin{align*}
		\prod_{2<p\leq z}\left(\frac{p-1}{p-2}\right)=\frac{2e^{-\gamma}A_0+o(1)}{\log z},
	\end{align*}
	where
	\begin{align*}
		A_0=\prod_{p>2}\left(1-\frac{1}{(p-1)^2}\right).
	\end{align*}
	Then we have
	\begin{align*}
		M_j\leq (2A_0+o(1))\frac{x}{\log D}\sum_{y^\delta<p_1,\cdots,p_j\leq y}\sum_{\substack{ \ell\leq x^{1-c}(2p_1\cdots p_j)^{-1}\\(\ell,P(y^\delta,y))=1}}\frac{H(p_1\cdots p_j\ell)}{p_1\cdots p_j\ell\log(x/(2p_1\cdots p_j\ell))},\tag{3.5}
	\end{align*}
	Summing over all $1\leq j\leq J$, by partial summation and Lemma 2.8, we have
	\begin{align*}
		\sum_{1\leq j\leq J}\frac{M_j}{j!}&\leq (2A_0+o(1))\frac{x}{\log D}\sum_{1\leq j\leq J}\frac{1}{j!}\sum_{y^\delta<p_1,\cdots,p_j\leq y}\sum_{\substack{ \ell\leq x^{1-c}/2\\p_1\cdots p_j|\ell\\(\ell(p_1\cdots p_j)^{-1},P(y^\delta,y))=1}}\frac{H(\ell)}{\ell\log(x/(2\ell))}\\&\leq (2A_0+o(1))\frac{x}{\log D}\sum_{\substack{ \ell\leq x^{1-c}/2}}\frac{H(\ell)}{\ell\log(x/(2\ell))}\\&=(2+o(1))\frac{x}{\log D}\int_1^{x^{1-c}/2}\frac{1}{u\log(x/(2u))}\mathrm{d}u\\&=\left(\frac{2}{5/8-\varepsilon}\log\frac{1}{c}+o(1)\right)\frac{x}{\log x}.\tag{3.6}
	\end{align*}
	Similar to (3.4) and (3.5), for a suitable absolute constant $c_1$, we have
	\begin{align*}
		&\sum_{y^\delta<p_1,\cdots,p_{J+1}\leq y}\sum_{\substack{ \ell\leq x^{1-c}\\p_1\cdots p_{J+1}|\ell\\2|\ell }}\sum_{\substack{\ell p\leq x}}\sum_{d|(\ell p+1,P(z))}\lambda_d^+\\&\leq(2A_0+o(1))\frac{x}{\log D} \sum_{y^\delta<p_1,\cdots,p_{J+1}\leq y}\sum_{\substack{ \ell\leq x^{1-c}(2p_1\cdots p_{J+1})^{-1} }}\frac{H(\ell)}{p_1\cdots p_{J+1}\ell\log(x/(2p_1\cdots p_{J+1}\ell))}\\&\leq  (c_1+o(1))\frac{x}{\log x} \sum_{y^\delta<p_1,\cdots,p_{J+1}\leq y}\frac{1}{p_1\cdots p_{J+1}}\\&= (c_1+o(1))\frac{x}{\log x} \left(\log\frac{1}{\delta}\right)^{J+1}.
	\end{align*}
	
	To summarize all the above, we obtain
	\begin{align*}
		\mathscr{S}_1\leq \left(\frac{2}{5/8-\varepsilon}\log\frac{1}{c}+c_1\delta+o(1)\right)\pi(x).\tag{3.7}
	\end{align*}

	\subsection{Estimate of $\mathscr{S}_2$} We will show that for sufficiently small $\delta$, $\mathscr{S}_2$ is small.
	Now we choose $D=z^2=x^{1/2}/(\log x)^{B}$.
	We have
	\begin{align*}
		\mathscr{S}_2&=\sum_{\substack{ \ell\leq x^{1-c}\\(\ell,P(y^\delta,y))=1\\2|\ell }}\sum_{\substack{\ell p\leq x\\\ell p+1\ \text{is prime}}}1\leq \sum_{\substack{ (\log x)^{2B}<\ell\leq x^{1-c}\\(\ell,P(y^\delta,y))=1\\2|\ell }}\sum_{\substack{\ell p\leq x\\(\ell p+1,P(z))=1}}1+O(z)\\&\leq \sum_{\substack{ (\log x)^{2B}<\ell\leq x^{1-c}\\(\ell,P(y^\delta,y))=1\\2|\ell }}\sum_{\substack{\ell p\leq x}}\sum_{d|(\ell p+1,P(z))}\lambda_d^++o(\pi(x))\\&= \sum_{d|P(z)}\lambda_d^+\sum_{\substack{ (\log x)^{2B}<\ell\leq x^{1-c}\\(\ell,P(y^\delta,y))=1\\2|\ell }}\sum_{\substack{\ell p\leq x\\\ell p+1\equiv 0\mkern-15mu\pmod{d}}}1+o(\pi(x))\\&=\sum_{d|P(z)}\frac{\lambda_d^+}{\varphi(d)}\sum_{\substack{ (\log x)^{2B}<\ell\leq x^{1-c}\\(\ell,dP(y^\delta,y))=1\\2|\ell }}\pi(x/\ell)\\&\quad+\sum_{d|P(z)}\lambda_d^+\sum_{\substack{ (\log x)^{2B}<\ell\leq x^{1-c}\\(\ell,dP(y^\delta,y))=1\\2|\ell }}\Bigg(\sum_{\substack{\ell p\leq x\\\ell p+1\equiv 0\mkern-15mu\pmod{d}}}1-\frac{\pi(x/\ell)}{\varphi(d)}\Bigg)+o(\pi(x)).\tag{3.8}
	\end{align*}
	By Lemma 2.7, we can bound the second summation by
	\begin{align*}
		\ll \frac{x}{(\log x)^A},
	\end{align*}
	which is negligible.
	For the main term, for a suitable absolute constant $c_2$, we have\begin{align*}
		\sum_{d|P(z)}\frac{\lambda_d^+}{\varphi(d)}\sum_{\substack{ \ell\leq x^{1-c}\\(\ell,dP(y^\delta,y))=1\\2|\ell }}\pi(x/\ell)&=\sum_{\substack{\ell \leq x^{1-c}\\(\ell, P(y^\delta,y))=1\\2|\ell}}\pi(x/\ell)\sum_{d|P'(z)}\frac{\lambda_d^+}{\varphi(d)}\\&\leq (F(2)+o(1))\sum_{\substack{\ell \leq x^{1-c}\\(\ell, P(y^\delta,y))=1\\2|\ell}}\pi(x/\ell)\prod_{\substack{p\leq z\\(p,\ell)=1}}\left(1-\frac{1}{p-1}\right)\\&\leq (c_2+o(1))\frac{x}{(\log x)^2}\sum_{\substack{\ell \leq x^{1-c}\\(\ell, P(y^\delta,y))=1}}\frac{H(\ell)}{\ell}\,\tag{3.9}
	\end{align*}
	where $P'(z)=\sum_{p\leq z, (p,\ell)=1}p$. 
	Write $\ell=\ell_1\ell_2$ with $P^+(\ell_1)\leq y^\delta$ and $P^-(\ell_2)>y$.
	If $\ell_2=1$, we have
	\begin{align*}
		\sum_{\substack{\ell \leq x^{1-c}\\P^+(\ell)\leq y^\delta}}\frac{H(\ell)}{\ell}&=\sum_{\substack{\ell \leq y^\delta}}\frac{H(\ell)}{\ell}+\sum_{\substack{y^\delta<\ell \leq x^{1-c}\\P^+(\ell)\leq y^\delta}}\frac{H(\ell)}{\ell}\\&\leq (A_0^{-1}\delta+o(1))\log y+(A_0^{-1}+o(1))\int_{y^\delta}^{x^{1-c}}\frac{1}{t}\mathrm{d}\Bigg(t\rho\left(\frac{\log t}{\log y^\delta}\right)\Bigg)\\&\leq \Bigg(A_0^{-1}\left(1+\int_{1}^{\infty}\rho(u)\mathrm{d}u\right)\delta+o(1)\Bigg)\log y
	\end{align*}
	Putting $l_2$ in dynadic range $\ell_2\sim L_2$ with $y<L_2<x^{1-c}$, we have
	\begin{align*}
		\sum_{\substack{\ell \leq x^{1-c}\\(\ell, P(y,x^{1-c}))>1\\(\ell, P(y^\delta,y))=1}}\frac{H(\ell)}{\ell}\leq \frac{(1-c)\log x}{\log 2} \max_{y\leq L_2<x^{1-c}}\sum_{\substack{\ell_2\sim L_2\\P^-(\ell_2)>y}}\frac{1}{L_2}\sum_{\substack{\ell_1\leq x^{1-c}(2L_2)^{-1}\\P^+(\ell_1)\leq y^\delta}}\frac{H(\ell_1)}{\ell_1}.
	\end{align*}
	If $x^{1-c}y^{-\delta}<L_2<x^{1-c}$, we have
	\begin{align*}
		\sum_{\substack{\ell_2\sim L_2\\P^-(\ell_2)>y}}\frac{1}{L_2}\sum_{\substack{\ell_1\leq x^{1-c}(L_2)^{-1}\\P^+(\ell_1)\leq y^\delta}}\frac{H(\ell_1)}{\ell_1}\leq\sum_{\substack{\ell_2\sim L_2\\P^-(\ell_2)>y}}\frac{1}{L_2}\sum_{\substack{\ell_1\leq y^\delta}}\frac{H(\ell_1)}{\ell_1}\leq A_0^{-1}\delta+o(1).
	\end{align*}
	If $y<L_2\leq x^{1-c}y^{-\delta}$, we have
	\begin{align*}
		\sum_{\substack{\ell_2\sim L_2\\P^-(\ell_2)>y}}\frac{1}{L_2}\sum_{\substack{\ell_1\leq x^{1-c}(L_2)^{-1}\\P^+(\ell_1)\leq y^\delta}}\frac{H(\ell_1)}{\ell_1}&\leq \sum_{\substack{\ell_2\sim L_2\\P^-(\ell_2)>y}}\frac{1}{L_2}\sum_{\substack{\ell_1\leq y^\delta}}\frac{H(\ell_1)}{\ell_1}+\sum_{\substack{\ell_2\sim L_2\\P^-(\ell_2)>y}}\frac{1}{L_2}\sum_{\substack{y^\delta<\ell_1\leq x}}\frac{H(\ell_1)}{\ell_1}\\&\leq A_0^{-1}\delta+o(1)+\frac{A_0^{-1}+o(1)}{\log y}\int_{y^\delta}^{x}\frac{1}{t}\mathrm{d}\Bigg(t\rho\left(\frac{\log t}{\log y^\delta}\right)\Bigg)\\&\leq  A_0^{-1}\left(1+\int_{1}^{\infty}\rho(u)\mathrm{d}u\right)\delta+o(1).
	\end{align*}
	Finally, for a suitable absolute constant $c_3$, we have
	\begin{align*}
		\mathscr{S}_2&\leq \Bigg(c_3\delta+o(1)\Bigg)\pi(x).\tag{3.10}
	\end{align*}
	
	\subsection{Conclusion}
	Conclude from (3.1), (3.7) and (3.10) that for any $e^{-5/16}<c<1$, we have
	\begin{align*}
		\mathop{\lim\sup}_{x\rightarrow \infty}\frac{T_c'(x)}{\pi(x)}\leq \frac{2\log (1/c)}{5/8-\varepsilon}+\left(c_1+c_3\right)\delta,
	\end{align*}
	for any $\varepsilon>0$ and $0<\delta<1/100$, $c_1$ and $c_3$ are two suitable absolute constants. As $\max(\varepsilon,\delta)\rightarrow 0^+$, we complete the proof of Theorem 1.1.

	\section{Proof of Theorem 1.3}
	In this section, we follow the work in [32], and we can obtain a better level for the remainder term.
	Following the work in [32, after (4.9)], we need to give an upper bound of the expression
	\begin{align*}
		S=&\sum_{x^{1-\eta_1}<p'<x^{1-\eta_1+t_1+t_2}}\left(\log x-\frac{1}{1-\eta_1+t_1+t_2}\log p'\right)\sum_{x^{1-\eta_1}<p<x^{1-\eta_2}}\sum_{\substack{dp\leq x\\dp-1\equiv0\mkern-15mu\pmod{p'}}}1\\\leq&\log x\sum_{x^{\eta_1-t_1-t_2}/(\log x)^B<m<x^{\eta_1}}f_1(m)|\{dp\in \mathcal{B}(m):(dp-1)/m\  \text{is a prime number}\}|\\&+O\left(\frac{x}{(\log x)^A}\right),\tag{4.1}
	\end{align*}
	where $0<\eta_2<\eta_1<1/2$, $t_1+t_2<\eta_1$, $\min(t_1,t_2)>0$,
	\begin{align*}
		f_1(m)=1-\frac{1}{1-\eta_1+t_1+t_2}\frac{\log(x/(m(\log x)^B))}{\log x},
	\end{align*}
	\begin{align*}
			\mathcal{B}(m)=\{dp\leq x:x^{1-\eta_1}<p<x^{1-\eta_2},dp-1\equiv0\mkern-15mu\pmod{m}\}.
	\end{align*}
	Put $m\in(x^{\eta_1-t_1-t_2}/(\log x)^B,x^{\eta_1})$ in dynadic range $m\sim\mathcal{M}$. We have
	\begin{align*}
		&\sum_{m\sim \mathcal{M}}f_1(m)|\{dp\in \mathcal{B}(m):(dp-1)/m\  \text{is a prime number}\}|\\&\leq \sum_{m\sim \mathcal{M}}f_1(m)\sum_{\substack{x^{\eta_2}/(\log x)^B<d<x^{\eta_1}\\2|dm}}\sum_{\substack{p\leq x/d\\(dp-1)/m\ \text{is prime}}}1+O\left(\frac{x}{(\log x)^A}\right)\\&\leq \sum_{\substack{m\sim \mathcal{M}\\2|m}}f_1(m)\sum_{\substack{x^{\eta_2}/(\log x)^B<d<x^{\eta_1}}}\sum_{\substack{p\leq x/d\\(dp-1)/m\ \text{is prime}}}1\\&\quad + \sum_{m\sim \mathcal{M}}f_1(m)\sum_{\substack{x^{\eta_2}/(\log x)^B<d<x^{\eta_1}\\2|d}}\sum_{\substack{p\leq x/d\\(dp-1)/m\ \text{is prime}}}1+O\left(\frac{x}{(\log x)^A}\right).\tag{4.2}
	\end{align*}
	\subsection{Case 1: $0<\eta_1<3/8$}
	Let $D=z^2=x^{5/8-\varepsilon}(2\mathcal{M})^{-1}$.
	Fix $0<\delta<1/100$. Let $y=x^{\eta_2\varepsilon/(3J)}$, where $J=J(\delta)$ is a parameter satisfying
	\begin{align*}
		\frac{(\log(1/\delta))^{J+1}}{(J+1)!}\leq \delta.
	\end{align*}
	Similar to (3.3), there exists a sequence $\{\lambda_{d'}^+\}_{d'=1}^{\infty}$ of real numbers, vanishing for $d'>D$ or $\mu(d')=0$, satisfying $\lambda_1^+=1$, $|\lambda_{d'}^+|=O(1)$, $0\leq \mu*1\leq \lambda^+*1$ such that
	\begin{align*}
		\sum_{\substack{x^{\eta_2}/(\log x)^B<d<x^{\eta_1}}}&\sum_{\substack{p\leq x/d\\(dp-1)/m\ \text{is prime}}}1\leq \sum_{\substack{x^{\eta_2}/(\log x)^B<d<x^{\eta_1}}}\sum_{\substack{p\leq x/d\\((dp-1)/m,P(z))=1}}1\\&\leq \sum_{\substack{x^{\eta_2}/(\log x)^B<d<x^{\eta_1}}}\sum_{\substack{p\leq x/d}}\sum_{d'|((dp-1)/m,P(z))}\lambda^+_{d'}\\&\leq \sum_{1\leq j\leq J}\sum_{y^\delta<p_1,\cdots,p_j\leq y}\sum_{\substack{x^{\eta_2}/(\log x)^B<d<x^{\eta_1}\\p_1\cdots p_j|d\\(d(p_1\cdots p_j)^{-1}, P(y^\delta,y))=1}}\sum_{\substack{p\leq x/d}}\sum_{d'|((dp-1)/m,P(z))}\lambda^+_{d'}\\&\quad+\sum_{y^\delta<p_1,\cdots,p_{J+1}\leq y}\sum_{\substack{x^{\eta_2}/(\log x)^B<d<x^{\eta_1}\\p_1\cdots p_{J+1}|d}}\sum_{\substack{p\leq x/d}}\sum_{d'|((dp-1)/m,P(z))}\lambda^+_{d'}\\&\quad+\sum_{\substack{x^{\eta_2}/(\log x)^B<d<x^{\eta_1}\\(d,P(y^\delta,y))=1}}\sum_{\substack{p\leq x/d}}\sum_{d'|((dp-1)/m,P(z))}\lambda^+_{d'}+O\left(\frac{x}{my^\delta}\right),\tag{4.3}
	\end{align*}
	where the term $O(\cdot)$ is neligiable.
	Note that for a well-factorable function $\lambda_{d'}=\lambda(d')$, we can write
	\begin{align*}
		\nu*\lambda=\nu*\lambda_2*\lambda_3=(\nu*\lambda_3)*\lambda_2,
	\end{align*}
	where $\nu=1_{(\mathcal{M},2\mathcal{M}]}f_1$, $\lambda=\lambda_2*\lambda_3$, $\lambda_2(q_2)$ is supported on $q_2\leq x^{1/4-\varepsilon}$, $\lambda_3(q_3)$ is supported on $q_3\leq x^{3/8}(2\mathcal{M})^{-1}$.
	Then by Lemma 3.2, for any $A>1$, we have
	\begin{align*}
		\sum_{m\sim \mathcal{M}}f_1(m)\sum_{d'}\lambda_{d'}\sum_{\substack{p_1\in\mathscr{P}_1}}\sum_{e\in\mathscr{E}}b_e\sum_{p\in\mathscr{P}}\Bigg(1_{p_1ep\equiv a\mkern-15mu\pmod{md'}}-\frac{1_{(p_1ep,md')=1}}{\varphi(md')}\Bigg)\ll\frac{x}{(\log x)^A},
	\end{align*}
	where $\lambda_{d'}$ is a well factorable function of level $D$, $|b_e|\leq 1$, $\mathscr{P}_1=[(1-\Delta)P_1,P_1]$, $\mathscr{E}=[(1-\Delta)E,E]$, $\mathscr{P}=[(1-\Delta)P,P]$ with $P_1EP\asymp x$, $x^{\eta_2}/(\log x)^B<P_1E< x^{\eta_1}$,  $y^\delta<P_1\leq y$, $\Delta=(\log x)^{-A_1}$. ($A_1$ is a suitable constant.) Thus we can bound the remainder term in the following expression
	\begin{align*}
		&\sum_{m\sim \mathcal{M}}f_1(m)\sum_{y^\delta<p_1,\cdots,p_j\leq y}\sum_{\substack{x^{\eta_2}/(\log x)^B<d<x^{\eta_1}\\p_1\cdots p_j|d\\(d(p_1\cdots p_j)^{-1}, P(y^\delta,y))=1}}\sum_{\substack{p\leq x/d}}\sum_{d'|((dp-1)/m,P(z))}\lambda^+_{d'}\\&=\sum_{m\sim \mathcal{M}}f_1(m)\sum_{d'|P(z)}\lambda^+_{d'}\sum_{y^\delta<p_1,\cdots,p_j\leq y}\sum_{\substack{x^{\eta_2}/(\log x)^B<d<x^{\eta_1}\\p_1\cdots p_j|d\\(d(p_1\cdots p_j)^{-1}, P(y^\delta,y))=1}}\sum_{\substack{p\leq x/d\\dp\equiv 1\mkern-15mu\pmod{md'}}}1\\&=\sum_{m\sim \mathcal{M}}f_1(m)\sum_{d'|P(z)}\lambda^+_{d'}\sum_{y^\delta<p_1,\cdots,p_j\leq y}\sum_{\substack{x^{\eta_2}/(\log x)^B<d<x^{\eta_1}\\p_1\cdots p_j|d\\(d(p_1\cdots p_j)^{-1}, P(y^\delta,y))=1\\(d,md')=1}}\frac{\pi(x/d)}{\varphi(md')}\\&\quad+\sum_{m\sim \mathcal{M}}f_1(m)\sum_{d'|P(z)}\lambda^+_{d'}\sum_{y^\delta<p_1,\cdots,p_j\leq y}\sum_{\substack{x^{\eta_2}/(\log x)^B<d<x^{\eta_1}\\p_1\cdots p_j|d\\(d(p_1\cdots p_j)^{-1}, P(y^\delta,y))=1\\(d,md')=1}}\Bigg(\sum_{\substack{p\leq x/d\\dp\equiv 1\mkern-15mu\pmod{md'}}}1-\frac{\pi(x/d)}{\varphi(md')}\Bigg)\\&=\sum_{m\sim \mathcal{M}}f_1(m)\sum_{d'|P(z)}\lambda^+_{d'}\sum_{y^\delta<p_1,\cdots,p_j\leq y}\sum_{\substack{x^{\eta_2}/(\log x)^B<d<x^{\eta_1}\\p_1\cdots p_j|d\\(d(p_1\cdots p_j)^{-1}, P(y^\delta,y))=1\\(d,md')=1}}\frac{\pi(x/d)}{\varphi(md')}+O\left(\frac{x}{(\log x)^A}\right).
	\end{align*}
	For the main term, by Lmma 2.6, we have
	\begin{align*}
		&\sum_{d'|P(z)}\lambda^+_{d'}\sum_{y^\delta<p_1,\cdots,p_j\leq y}\sum_{\substack{x^{\eta_2}/(\log x)^B<d<x^{\eta_1}\\p_1\cdots p_j|d\\(d(p_1\cdots p_j)^{-1}, P(y^\delta,y))=1\\(d,md')=1}}\frac{\pi(x/d)}{\varphi(md')}\\&=\sum_{y^\delta<p_1,\cdots,p_j\leq y}\sum_{\substack{x^{\eta_2}/(\log x)^B<d<x^{\eta_1}\\p_1\cdots p_j|d\\(d(p_1\cdots p_j)^{-1}, P(y^\delta,y))=1\\(d,m)=1}}\pi(x/d)\sum_{d'|P'(z)}\frac{\lambda_{d'}^+}{\varphi(md')}\\&\leq (F(2)+o(1))\sum_{y^\delta<p_1,\cdots,p_j\leq y}\sum_{\substack{x^{\eta_2}/(\log x)^B<d<x^{\eta_1}\\p_1\cdots p_j|d\\(d(p_1\cdots p_j)^{-1}, P(y^\delta,y))=1\\(d,m)=1}}\frac{\pi(x/d)}{\varphi(m)}\\&\quad\times\prod_{2<p\leq z,(p,dm)=1}\left(\frac{p-2}{p-1}\right)\prod_{2<p\leq z,p|m}\left(\frac{p-1}{p}\right),
	\end{align*}
	where $P'(z)=\prod_{p\leq z,(p,d)=1}p$. 
	Summing over all $1\leq j\leq J$, we have
	\begin{align*}
		&\sum_{1\leq j\leq J}\sum_{y^\delta<p_1,\cdots,p_j\leq y}\sum_{\substack{x^{\eta_2}/(\log x)^B<d<x^{\eta_1}\\p_1\cdots p_j|d\\(d(p_1\cdots p_j)^{-1}, P(y^\delta,y))=1\\(d,m)=1}}\frac{\pi(x/d)}{\varphi(m)}\prod_{2<p\leq z,(p,dm)=1}\left(\frac{p-2}{p-1}\right)\prod_{2<p\leq z,p|m}\left(\frac{p-1}{p}\right)\\&\leq \sum_{\substack{x^{\eta_2}/(\log x)^B<d<x^{\eta_1}\\(d,m)=1}}\frac{\pi(x/d)}{m}\prod_{2<p\leq z,(p,dm)=1}\left(\frac{p-2}{p-1}\right).\tag{4.4}
	\end{align*}
	We may then follow the proof in [32], merely replacing \(4/7\) (or \(1/2\)) with \(5/8\).
	
	We now need to show that the contributions from the remaining terms are small.
	For a suitable absolute constant $c_1$, omitting the details, we have
	\begin{align*}
		&\sum_{m\sim \mathcal{M}}f_1(m)\sum_{y^\delta<p_1,\cdots,p_{J+1}\leq y}\sum_{\substack{x^{\eta_2}/(\log x)^B<d<x^{\eta_1}\\p_1\cdots p_{J+1}|d}}\sum_{\substack{p\leq x/d}}\sum_{d'|((dp-1)/m,P(z))}\lambda^+_{d'}\\&=\sum_{m\sim \mathcal{M}}f_1(m)\sum_{d'|P(z)}\lambda^+_{d'}\sum_{y^\delta<p_1,\cdots,p_{J+1}\leq y}\sum_{\substack{x^{\eta_2}/(\log x)^B<d<x^{\eta_1}\\p_1\cdots p_{J+1}|d\\(d,md')=1}}\frac{\pi(x/d)}{\varphi(md')}+O\left(\frac{x}{(\log x)^A}\right)\\&= \sum_{m\sim \mathcal{M}}f_1(m)\sum_{y^\delta<p_1,\cdots,p_{J+1}\leq y}\sum_{\substack{x^{\eta_2}/(\log x)^B<d<x^{\eta_1}\\p_1\cdots p_{J+1}|d\\(d,m)=1}}\pi(x/d)\sum_{d'|P'(z)}\frac{\lambda_{d'}^+}{\varphi(md')}+O\left(\frac{x}{(\log x)^A}\right)\\&\leq (e^{\gamma}+o(1))\sum_{m\sim \mathcal{M}}\frac{f_1(m)}{m}\sum_{y^\delta<p_1,\cdots,p_{J+1}\leq y}\sum_{\substack{x^{\eta_2}/(\log x)^B<d<x^{\eta_1}\\p_1\cdots p_{J+1}|d\\(d,m)=1}}\pi(x/d)\prod_{2<p\leq z,(p,dm)=1}\left(\frac{p-2}{p-1}\right)\\&\quad+O\left(\frac{x}{(\log x)^A}\right)\\&\leq (2A_0+o(1))\frac{x}{\log D}\sum_{m\sim \mathcal{M}}\frac{f_1(m)H(m)}{m}\sum_{y^\delta<p_1,\cdots,p_{J+1}\leq y}\sum_{\substack{x^{\eta_2}/(\log x)^B<d<x^{\eta_1}\\p_1\cdots p_{J+1}|d}}\frac{H(d)}{d\log(x/d)}\\&\quad+O\left(\frac{x}{(\log x)^A}\right)\\&\leq \left(c_1\left(\log\frac{1}{\delta}\right)^{J+1}\log\frac{1-\eta_2}{1-\eta_1}+o(1)\right)\frac{x}{\log x}.\tag{4.5}
	\end{align*}
	since
	\begin{align*}
		\sum_{\substack{x^{\eta_2}/(\log x)^B<d<x^{\eta_1}\\p_1\cdots p_{J+1}|d}}\frac{H(d)}{d\log(x/d)}&=\frac{1+o(1)}{p_1\cdots p_{J+1}}\sum_{\substack{x^{\eta_2}/(\log x)^{B}<p_1\cdots p_{J+1}d<x^{\eta_1}}}\frac{H(d)}{d\log(x/(p_1\cdots p_Jd))}\\&=\frac{A_0^{-1}+o(1)}{p_1\cdots p_{J+1}}\log\frac{1-\eta_2}{1-\eta_1}.
	\end{align*}
	We also need to estimate the terms $(d,P(y^\delta,y))=1$. We are unable to use Lemma 3.2 directly here. By Lemma 2.7, set \(D=z^2=x^{1/2}(\log x)^{-B}(2\mathcal{M})^{-1}\). For a suitable absolute constant $c_2$, we have
	\begin{align*}
		&\sum_{m\sim \mathcal{M}}f_1(m)\sum_{\substack{x^{\eta_2}/(\log x)^B<d<x^{\eta_1}\\(d,P(y^\delta,y))=1}}\sum_{\substack{p\leq x/d}}\sum_{d'|((dp-1)/m,P(z))}\lambda^+_{d'}\\&\leq (2A_0+o(1))\frac{x}{\log D}\sum_{m\sim \mathcal{M}}\frac{f_1(m)H(m)}{m}\sum_{\substack{x^{\eta_2}/(\log x)^B<d<x^{\eta_1}\\(d,P(y^\delta,y))=1}}\frac{H(d)}{d\log(x/d)}+O\left(\frac{x}{(\log x)^A}\right)\\&\leq (c_2+o(1))\frac{x}{\log x}\sum_{\substack{x^{\eta_2}/(\log x)^B<d<x^{\eta_1}\\(d,P(y^\delta,y))=1}}\frac{H(d)}{d\log(x/d)}+O\left(\frac{x}{(\log x)^A}\right).\tag{4.6}
	\end{align*}
	Write $d=d_1d_2$ with $P^+(d_1)\leq y^{\delta}$ and $P^-(d_2)>y$. Putting $d_2$ in dynadic range $d_2\sim D_2$ with $y<D_2<x^{\eta_1}$, for a suitable absolute constant $c_3$, we have 
	\begin{align*}
		\sum_{\substack{x^{\eta_2}/(\log x)^B<d<x^{\eta_1}\\(d,P(y^\delta,y))=1}}&\frac{H(d)}{d\log(x/d)}\leq \sum_{\substack{x^{\eta_2}/(\log x)^B<d<x^{\eta_1}\\P^+(d)\leq y^\delta}}\frac{H(d)}{d\log(x/d)}\\&+\sum_{D_2}\sum_{\substack{d_2\sim D_2\\P^-(d_2)>y}}\frac{1}{D_2}\sum_{\substack{x^{\eta_2}/(d_2(\log x)^B)<d_1<x^{\eta_1}/d_2\\P^+(d_1)\leq y^{\delta}}}\frac{H(d_1)}{d_1\log(x/(d_1d_2))}\\&\leq \left(c_3\delta+o(1)\right)\log\frac{1-\eta_2}{1-\eta_1},
	\end{align*}
	since
	\begin{align*}
		&\sum_{D_2}\sum_{\substack{d_2\sim D_2\\P^-(d_2)>y}}\frac{1}{D_2}\sum_{\substack{x^{\eta_2}/(d_2(\log x)^B)<d_1<x^{\eta_1}/d_2\\P^+(d_1)\leq y^{\delta}}}\frac{H(d_1)}{d_1\log(x/(d_1d_2))}\\&\leq \sum_{y<D_2<x^{\eta_2}/(\log x)^B}\frac{1}{\log y}\Bigg(A_0^{-1}\rho\left(\frac{\eta_2\log x-\log D_2}{\delta\log y}\right)+o(1)\Bigg)\log\frac{1-\eta_2}{1-\eta_1}\\&\quad +\frac{1}{(1-\eta_1)\log x}\sum_{x^{\eta_2}/(\log x)^B\leq D_2<x^{\eta_1}}\frac{1}{\log y}\sum_{\substack{1<d_1<x^{\eta_1}/d_2\\P^+(d_1)\leq y^{\delta}}}\frac{H(d_1)}{d_1}\\&\leq \Bigg(3A_0^{-1}\delta\int_0^{\infty}\rho(u)\mathrm{d}u+o(1)\Bigg)\log\frac{1-\eta_2}{1-\eta_1}.
	\end{align*}
	To summarize all the above, for sufficiently small $\delta>0$, we obtain
	\begin{align*}
		S\leq \Bigg(\left(2\int_{\eta_1-t_1-t_2}^{\eta_1}\frac{1-(1-t)/(1-\eta_1+t_1+t_2)}{4/7-\varepsilon-t}\mathrm{d}t+10^{-6}\right)\log\frac{1-\eta_2}{1-\eta_1}+o(1)\Bigg)x\log x.\tag{4.7}
		\end{align*}
	
	Similarly, for the $\mathscr{S}_B$ in [32, Lemma 4.1], we can use $5/8$ instead of $4/7$ to get \begin{align*}
		\mathscr{S}_B= \sum_{\substack{n<x\\P^+(n)\geq P^+(n+1)>x^{1-c}}}1&\leq \left(2\int_0^c\int_{0}^{\eta}\frac{1}{(5/8-t)(1-\eta)}\mathrm{d}t\mathrm{d}\eta+\nu+o(1)\right)x\\&= \left(2\int_0^c\log\left(\frac{1-t}{1-c}\right)\frac{1}{5/8-t}\mathrm{d}t+\nu+o(1)\right)x
	\end{align*}
	We also can use $5/8$ instead of $1/2$ in [32, after (4.19)] and get
	\begin{align*}
		\mathcal{D}\leq \Bigg(2\int_0^{1-s}\frac{(1-t)/s-1}{5/8-t}\mathrm{d}t+10^{-6}+o(1)\Bigg)g_2(\alpha,\beta)x
	\end{align*}
	and
	\begin{align*}
		B_2&\geq \Bigg(\frac{1}{2}\left(\frac{1}{s}-1\right)-2\int_0^{1-s}\frac{(1-t)/s-1}{5/8-t}\mathrm{d}t-10^{-6}+o(1)\Bigg)g_2(\alpha,\beta)x\\&\geq (0.042+o(1))g_2(\alpha,\beta)x,
	\end{align*}
	by taking $s=0.852$. Then we have
	\begin{align*}
		\mathcal{R}\leq\Bigg(0.229\left(\log\frac{1}{2\delta_1}\right)^2+o(1)\Bigg)x,\quad
		\mathscr{S}'_5\leq \left(0.458\int_{\delta_1}^{1/2}\frac{\log{(2-2t)}}{t}\mathrm{d}t+o(1)\right)x.\tag{4.8}
	\end{align*}
	
	\subsection{Case 2: $3/8<\eta_1<1/2$}
	In this case, we obtain a better level than $x^{5/8-\varepsilon}$ to estimate $S$. Without loss of generality, we assume $3/8<\eta_2<\eta_1$. (We can assume $|\eta_1-\eta_2|\leq 10^{-6}$.)
	We consider the following type of estimation
	\begin{align*}
		\sum_{m\sim \mathcal{M}}\sum_{d'\leq D}f_1(m)\lambda_{d'}\Bigg(\sum_{\substack{d\sim x^{\eta}}}\sum_{\substack{p<x/d\\dp\equiv a\mkern-15mu\pmod{md'}}}1-\frac{1}{\varphi(md')}\sum_{\substack{d\sim x^{\eta}}}\sum_{\substack{p<x/d\\(dp,md')=1}}1\Bigg)\ll\frac{x}{(\log x)^A},
	\end{align*}
	where $2\mathcal{M}<x^{\eta-\varepsilon}$,  $\eta_2<\eta<\eta_1$, $\lambda$ is a well-factorable function of level $D=x^{1/4+\eta-\varepsilon}(2\mathcal{M})^{-1}$. (We ignore the terms $x^{\eta-\varepsilon}/2\leq \mathcal{M}<x^{\eta_1}$, since their contributions are small.)
	In fact, we only need to prove that for $Q\leq Mx^{-\varepsilon}$ and $R<x^{1/4-\varepsilon}$, for any $A>1$
	\begin{align*}
		\sum_{q\sim Q}\sum_{r\sim R}\gamma_q\delta_r\Bigg(\sum_{\substack{m\sim M}}\sum_{\substack{n\sim N\\mn\equiv a\mkern-15mu\pmod{qr}}}\beta_n-\frac{1}{\varphi(qr)}\sum_{\substack{m\sim M}}\sum_{\substack{n\sim N\\(mn,qr)=1}}\beta_n\Bigg)\ll\frac{x}{(\log x)^A},\tag{4.9}
	\end{align*}
	where $(\gamma_q),(\delta_r), (\beta_n)$ are $1$-bound sequences.
	Note that for a well-factorable function $\lambda_{d'}=\lambda(d')$, we can write
	\begin{align*}
		\nu*\lambda=\nu*\lambda_1*\lambda_2=(\nu*\lambda_1)*\lambda_2,
	\end{align*}
	where $\nu=1_{(\mathcal{M},2\mathcal{M}]}f_1$, $\lambda=\lambda_1*\lambda_2$, $\lambda_1(d_1)$ is supported on $d_1\leq x^{\eta-\varepsilon}(2\mathcal{M})^{-1}$, $\lambda_2(d_2)$ is supported on $d_2\leq x^{1/4-\varepsilon}$.

	Now we begin to prove (4.9). Following Section 12 in [2], we apply Lemma 2.2 instead of [2, Lemma 2] to get
	\begin{align*}
		\mathscr{D}(M,N,Q,R)\ll &x^\varepsilon\frac{M}{Q}\Bigg|\mathop{\sum_{q\sim Q}\sum_{r\sim R}}_{(qr,a)=1}\frac{\gamma_q\delta_r}{r}\sum_{\substack{n\sim N\\(n,qr)=1}}\beta_n\int\Phi\left(\frac{uqr}{Qr}\right)\sum_{h}\Psi\left(\frac{h}{H}\right)\\&\times e\left(-h\frac{uM}{Qr}\right)e\left(-ah\frac{\overline{n}}{qr}\right)\mathrm{d}u\Bigg|+O(N^{1/2}M^{1/2}x^{1/2-\varepsilon/2}),
	\end{align*}
	where $H\ll x^\varepsilon QRM^{-1}$,  $\Phi$ and $\Psi$ are some compactly supported functions with $\Phi^{(k)}\ll_kx^{k\varepsilon}$ (this has no effect) and $\Psi^{(k)}\ll_k1$ for $k\geq 0$. 
	Our aim is to prove $\mathscr{D}(M,N,Q,R)\ll x^{1-\varepsilon}$.
	The contribution of the term $O(\cdot)$ is acceptable. Then we have (compared with [2, before (12.2)])
	\begin{align*}
		\mathscr{D}(M,N,Q,R)\ll&x^{\varepsilon}\frac{M}{QR}\max_{u\asymp1}\sum_{q\sim Q}|\gamma_q|\sum_{\substack{n\sim N\\(n,q)=1}}|\beta(n)|\\&\times\Bigg|\sum_{h}\Psi\left(\frac{h}{H}\right)\sum_{r\sim R,(r,an)=1}\delta'_re\left(-h\frac{uM}{Qr}\right)e\left(-ah\frac{\overline{n}}{qr}\right)\Bigg|+x^{1-\varepsilon/2},
	\end{align*}
	where $|\delta_r'|\leq1$.
	Noting that 
	\begin{align*}
		\frac{\overline{n}}{qr}\equiv \frac{\overline{qr}}{n}+\frac{1}{nqr}\mkern-7mu\pmod{1},
	\end{align*}
	we have 
	\begin{align*}
		e\left(-ah\frac{\overline{n}}{qr}\right)=e\left(-ah\frac{\overline{qr}}{n}\right)+O\left(\frac{|aH|}{NQR}\right).
	\end{align*}
	The contribution of the term $O(|aH|/(NQR))$ is
	\begin{align*}
		\ll \frac{M}{QR}QNHR\frac{H}{NQR}\ll x^\varepsilon\frac{QR}{M},
	\end{align*}
	which is accpetable.
	Then we have
	\begin{align*}
		\mathscr{D}(M,N,Q,R)\ll&x^\varepsilon\frac{M}{QR}\max_{u\asymp1}\sum_{q\sim Q}\sum_{\substack{n\sim N\\(n,q)=1}}\\&\times\Bigg|\sum_{h}\Psi\left(\frac{h}{H}\right)\sum_{r\sim R,(r,an)=1}\delta'_re\left(-h\frac{uM}{Qr}\right)e\left(-ah\frac{\overline{qr}}{n}\right)\Bigg|+x^{1-\varepsilon/2}.
	\end{align*}
	By Cauchy's inequity, we have
	\begin{align*}
		\mathscr{D}(M,N,Q,R)\ll&x^\varepsilon MN^{1/2}Q^{-1/2}R^{-1}\mathscr{A}^{1/2}+x^{1-\varepsilon/2},\tag{4.10}
	\end{align*}
	where
	\begin{align*}
		\mathscr{A}=\sum_{q\sim Q}\sum_{\substack{n\sim N\\(n,q)=1}}\Bigg|\sum_{h}\Psi\left(\frac{h}{H}\right)\sum_{r\sim R,(r,an)=1}\delta'_re\left(-h\frac{uM}{Qr}\right)e\left(-ah\frac{\overline{qr}}{n}\right)\Bigg|^2.
	\end{align*}
	Then we appeal to Lemma 2.5 to infer the following lemma.
	\begin{lemma}
		Let $1\leq C,D,H,K,Q\ll x^{O(1)}$, $a\neq 0$ and $\varepsilon>0$. For $q\sim Q$ and $\omega_q\in\mathbb{R}/\mathbb{Z}$,
		define
		\begin{align*}
			\mathscr{B}(C,D,H,Q):=\sum_{c\sim C}\sum_{d\sim D}\Bigg|\sum_h\Psi\left(\frac{h}{H}\right)\sum_{q\sim Q}\beta_q e\left(h\omega_q\right)e\left(ah\frac{\overline{dq}}{c}\right)\Bigg|^2.
		\end{align*}
		Then we have \begin{align*}
			&\mathscr{B}(C,D,H,Q)\\&\ll x^\varepsilon CDHQ+ x^\varepsilon HQ^{1/2}(H+Q)^{1/2}\\&\quad\times\Bigg(D^2HQ^3+\left(1+\frac{C^2(H+Q)\max T_H(\omega_q)}{H^2Q^4}\right)^{7/32}C(C+DQ^2)(Q^2+HQ)\Bigg)^{1/2}.
		\end{align*}
	\end{lemma}
	\begin{proof}
		Squaring and changing the order of summation we represent $\mathscr{B}(C, D, H, Q)$ as 
		\begin{align*}
			\mathscr{B}(C, D, H, Q)\leq &\sum_{c}\sum_{d}\Phi_0\left(\frac{c}{C},\frac{d}{D}\right)\sum_{h_1,h_2}\Psi\left(\frac{h_1}{H}\right)\overline{\Psi\left(\frac{h_2}{H}\right)}\sum_{q_1,q_2\sim Q}\beta_{q_1}\overline{\beta_{q_2}} \\&\times e\left(h_1\omega_{q_1}-h_2\omega_{q_2}\right) e\left(a(h_1q_2-h_2q_1)\frac{\overline{dq_1q_2}}{c}\right).
		\end{align*}
		Let $e=a(h_1q_2-h_2q_1)$. The contribution of $e=0$ is 
		\begin{align*}
			\mathscr{B}(e= 0)&\ll x^\varepsilon\sum_{c\sim C}\sum_{d\sim D}\sum_{h_1\asymp H}\sum_{q_2\sim Q}1\\&\ll x^\varepsilon CDHQ.\tag{4.11}
		\end{align*}
		Now consider $e\neq 0$. Without loss of generality, we assume that \(e>0\). Let $q_0=(q_1,q_2)$. Setting $q_1\leftarrow q_0q_1$, $q_2\leftarrow q_0q_2$, $e\leftarrow q_0e$, we have 
		\begin{align*}
			&\mathscr{B}(e\neq 0)\\&\ll x^\varepsilon K\max_{\substack{Q_0\ll Q}}\max_{\substack{G\asymp Q^2Q_0^{-1}\\E\ll|a|HQQ_0^{-1}}}\sum_{g\sim G}\max_{\substack{q_0q_1q_2=g\\q_1,q_2\sim Q/q_0\\(q_1,q_2)=1}}\Bigg|\sum_{e\sim E}a_{e,g}\sum_{c}\sum_{d}\Phi_0\left(\frac{c}{C},\frac{d}{D}\right)e\left(e\frac{\overline{dg}}{c}\right)\Bigg|,
		\end{align*}
		where \begin{align*}
			a_{e,g}=\sum_{\substack{a(h_1q_2-h_2q_1)=e}}\Psi\left(\frac{h_1}{H}\right)\overline{\Psi\left(\frac{h_2}{H}\right)}e\left(h_1\omega_{q_0q_1}-h_2\omega_{q_0q_2}\right),\quad w_g=1
		\end{align*}
		if the maximum is attained at some $q_1=q_1(g)$, $q_2=q_2(g)$; if the maximum is empty, we let $a_{e,g}=0$ and $w_g=0$. By Lemma 2.4, we get that the tuple $(g,E, x, (a_{e,g})_{e\sim E}, A_g, Y)$ satisfies Assumption 2.3, where
		\begin{align*}
			A_g&=\|(a_{e,g})_{e\sim E}\|_2+\sqrt{E\left(\frac{HQ_0}{Q}+\frac{H^2Q_0^2}{Q^2}\right)},\\ Y&=\frac{EHQ_0}{|a|(H+QQ_0^{-1})Q\max T_H(\omega_{q})}
		\end{align*}
		By Lemma 2.5 and (4.11), we get
		\begin{align*}
			\mathscr{B}(C,D&,H,Q)\ll x^\varepsilon CDHQ\\&+\max_{\substack{Q_0,E,G}} x^\varepsilon \|w_gA_g\|_2\Bigg(D^2EG+\left(1+\frac{C^2}{G^2Y}\right)^{2\theta_{\max}}C(C+DG)(G+E)\Bigg)^{1/2}.
		\end{align*}
		For $\|w_gA_g\|_2$, we have \begin{align*}
			&\|w_gA_g\|_2^2\\&\ll\sum_{e\sim E}\sum_{g\sim G}|a_{e,g}|^2+\sum_{g\sim G}E\left(\frac{HQ_0}{Q}+\frac{H^2Q_0^2}{Q^2}\right)
			\\&\ll \sum_{q_0\sim Q_0}\sum_{\substack{q_1,q_2\sim Q/q_0\\(q_1,q_2)=1}}\sum_{\substack{e\sim E\\a|e}}\Bigg(\sum_{\substack{h_1,h_2\asymp H\\h_1q_2-h_2q_1=e/a}}1\Bigg)^2+EG\left(\frac{HQ_0}{Q}+\frac{H^2Q_0^2}{Q^2}\right)\\&\ll \sum_{q_0\sim Q_0}\sum_{e\sim E,a|e}\sum_{\substack{q_1,q_2\sim Q/q_0\\(q_1,q_2)=1}}\sum_{\substack{h_1,h_2\asymp H\\h_1q_2-h_2q_1=e/a}}\sum_{\substack{h_1',h_2'\asymp H\\q_1(h_2-h_2')=q_2(h_1-h_1')}}1+EG\left(\frac{HQ_0}{Q}+\frac{H^2Q_0^2}{Q^2}\right)\\&\ll \sum_{q_0\sim Q_0}\sum_{e\sim E,a|e}
			\sum_{\substack{q_1\sim Q/q_0\\h_2\asymp H}}\sum_{\substack{q_2\sim Q/q_0\\h_1\asymp H\\h_1q_2=h_2q_1+e/a}}\sum_{\substack{h_1'\asymp H\\h_1'\equiv h_1\mkern-15mu\pmod{q_1}}}1+EG\left(\frac{HQ_0}{Q}+\frac{H^2Q_0^2}{Q^2}\right)
			\\&\ll E(H^2Q_0+HQ).
		\end{align*}
		Thus we get
		\begin{align*}
			&\mathscr{B}(C,D,H,Q)\\&\ll x^\varepsilon CDHQ+x^\varepsilon \max_{\substack{Q_0\ll Q\\E\ll |a|HQQ_0^{-1}\\G\asymp Q^2Q_0^{-1}}}(E(H^2Q_0+HQ))^{1/2}\\&\quad\times\Bigg(D^2EG+\left(1+\frac{C^2}{G^2Y}\right)^{7/32}C(C+DG)(G+E)\Bigg)^{1/2}\\&\ll x^\varepsilon CDHQ+\max_{Q_0\ll Q}x^\varepsilon HQ^{1/2}Q_0^{-1/2}(HQ_0+Q)^{1/2}\Bigg(D^2HQ^3Q_0^{-2}\\&\quad+\left(1+\frac{C^2(HQ_0+Q)\max T_H(\omega_q)}{H^2Q^4}\right)^{7/32}C(C+DQ^2Q_0^{-1})(Q^2+HQ)Q_0^{-1}\Bigg)^{1/2}.
		\end{align*}
		This bound is seen to be non-increasing in the $Q_0\gg 1$ parameter. Taking $Q_0=1$, we complete the proof.
	\end{proof}

	Let $Q\leq Mx^{-\varepsilon}$.
	We apply Lemma 4.1 to bound $\mathscr{A}$ and get
	\begin{align*}
		\mathscr{A}&\ll \mathscr{B}(N,Q,H,R)
		\\&\ll x^\varepsilon HR\Bigg(Q^2HR^3+\left(1+\frac{N^2}{H^2R^3}\right)^{7/32}N(N+QR^2)R^2\Bigg)^{1/2}+x^\varepsilon NQHR\\&\ll x^\varepsilon\frac{QR^2}{M}\Bigg(\frac{Q^3R^4}{M}+\left(1+\frac{x^2}{Q^2R^5}\right)^{7/32}N(N+QR^2)R^2\Bigg)^{1/2}+x^\varepsilon\frac{NQ^2R^2}{M}.
	\end{align*}
	Consider $QR>x^{1/2-\varepsilon}$. (Otherwise we obtain the required result from the generalized Bombieri-Vinogradov theorem.) Returning to (4.10), provided
	\begin{align*}
		M^2R^4<x^{2-\varepsilon},\quad R^4<x^{1-\varepsilon},\quad R<Mx^{-\varepsilon},
	\end{align*}
	\begin{align*}
		R^{43/32}<x^{-7/32-\varepsilon}M^2,\quad R^{93/32}<x^{9/16}M^{7/16}x^{-\varepsilon}, 
	\end{align*}
	the above bound satisfies the requirements.
	In fact, for $M>x^{3/8}$, the conditions $R<Mx^{-\varepsilon}$, $R^{43/32}<x^{-7/32-\varepsilon}M^2$ and $R^{93/32}<x^{9/16}M^{7/16}x^{-\varepsilon}$ are implied by others. Thus we only need to provide
	\begin{align*}
		M>x^{3/8},\quad
		M^2R^4<x^{2-\varepsilon},\quad R^4<x^{1-\varepsilon}.\tag{4.12}
	\end{align*}
	Now we complete the proof of (4.9) and Theorem 1.4.
	
	Now following the work in [32, (4.10)-(4.14)], we can bound $S$ by (we use $1/4+\eta_1-10^{-6}-\varepsilon$ instead of $1/2$)
	\begin{align*}
		S\leq \Bigg(2\int_{\eta_1-t_1-t_2}^{\eta_1}\frac{1-(1-t)/(1-\eta_1+t_1+t_2)}{1/4+\eta_1-10^{-6}-\varepsilon-t}\mathrm{d}t\log\frac{1-\eta_2}{1-\eta_1}+o(1)\Bigg)x\log x.\tag{4.13}
	\end{align*}
	
	\subsection{Conclusion}
	Similar to [32, (4.15)], by (4.7) and (4.13), we need to find the largest $t_1=t_1(\eta_1)\in(0,\eta_1)$ such that for some $0<t_2<\eta_1-t_1$, one has
	\begin{align*}
		&\frac{1}{2}\left(\frac{1}{1-\eta_1+t_1}-\frac{1}{1-\eta_1+t_1+t_2}\right)>g_1(\eta_1,t_1,t_2)+2\times10^{-6}.
	\end{align*}
	where
	\begin{align*}
		g_1(\eta_1,t_1,t_2)=\left\{ \begin{aligned}
			&2\int_{\eta_1-t_1-t_2}^{\eta_1}\frac{1-(1-t)/(1-\eta_1+t_1+t_2)}{5/8-t}\mathrm{d}t, \quad&& \eta_1\leq3/8,\\&2\int_{\eta_1-t_1-t_2}^{\eta_1}\frac{1-(1-t)/(1-\eta_1+t_1+t_2)}{1/4+\eta_1-10^{-6}-t}\mathrm{d}t,\quad&& 3/8<\eta_1<1/2.
		\end{aligned}\right.
	\end{align*}
	For $\eta=\eta_1$, by numerical calculation, we show the pragh of $t_1(\eta)$:
	\begin{figure}[ht]
		\centering
		\includegraphics[width=0.9\textwidth]{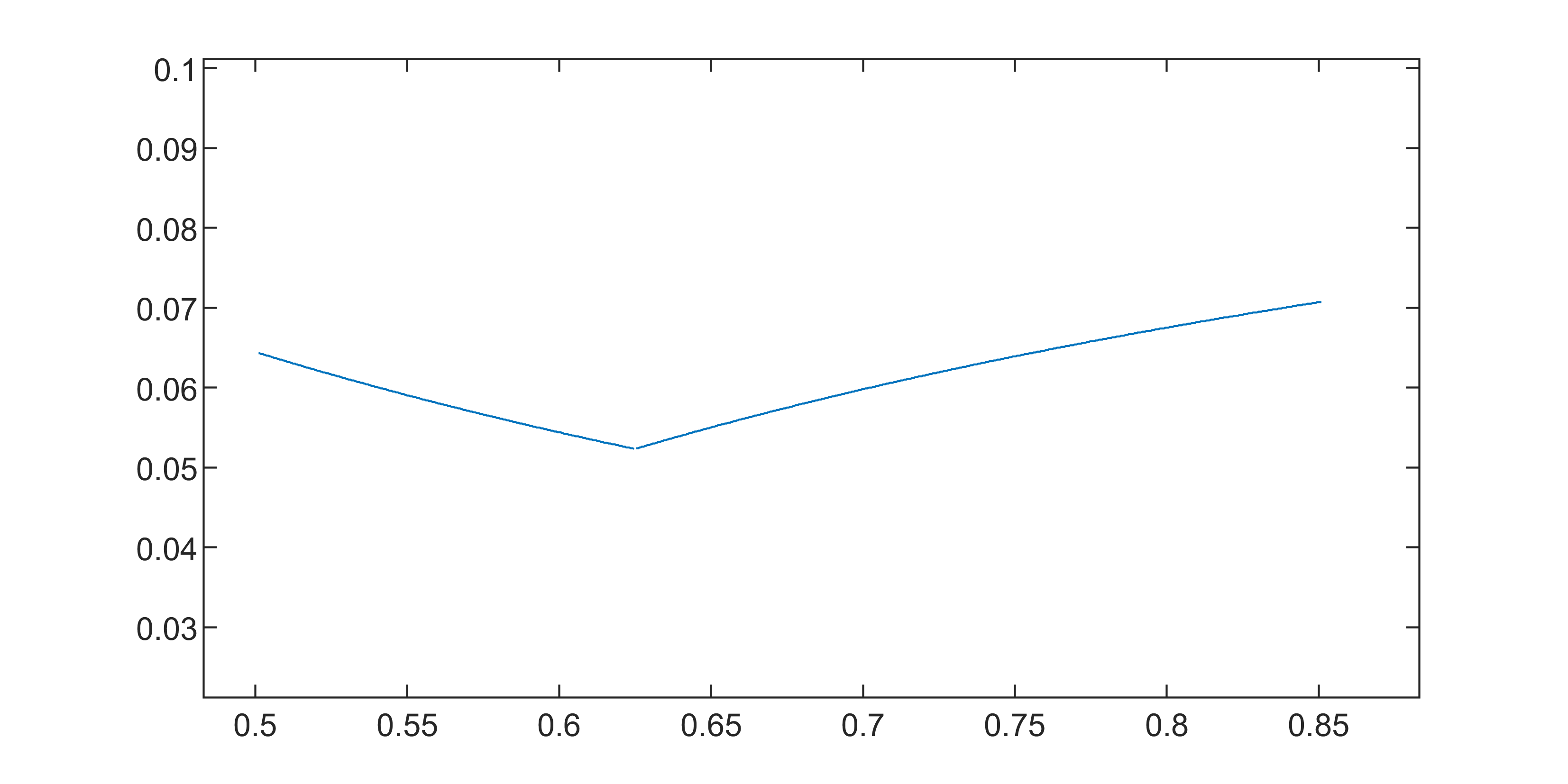} % 替换为你的图片文件名
		\caption{$\eta\rightarrow t(\eta)=t_1(1-\eta)$}
	\end{figure}

	Finally, we have (compared with [32, \S 4.4])
	\begin{align*}
		\sum_{\substack{n<x\\P^+(n)<P^+(n+1)}}1\geq (\mathcal{C}(c,\delta_1)-2\times 10^{-6}+o(1))x,
	\end{align*} 
	where $\mathcal{C}(c,\delta_1)$ is defined by
	\begin{align*}
		\mathcal{C}(c,\delta_1):=&\log\frac{1}{1-c}-2\int_0^c\log\left(\frac{1-t}{1-c}\right)\frac{\mathrm{d}t}{5/8-t}+\int_{\delta_1}^{1/2}\rho\left(\frac{1}{t}\right)\frac{1}{t}\mathrm{d}t\\&+\int_c^{1/2}\frac{1-2\eta+2t_1(\eta)}{2(1-\eta)(1-\eta+t_1(\eta))}\mathrm{d}\eta\\&-0.229\Bigg(\log\frac{1}{2\delta_1}\Bigg)^2-0.458\int_{\delta_1}^{1/2}\frac{\log{(2-2t)}}{t}\mathrm{d}t
	\end{align*}
	with $c=0.149$, $\delta_1=0.417$. By numerical calculation, we have
	\begin{align*}
		\mathcal{C} (c,\delta_1)-2\times 10^{-6}> 0.299.
	\end{align*}
	We complete the proof of Theorem 1.3.

	\section*{Acknowledegements}

\end{document}